\documentclass[11pt,a4paper]{article}
\usepackage[utf8]{inputenc}
\usepackage[T1]{fontenc}
\usepackage{amsmath, amssymb, amsthm, mathrsfs}
\usepackage{graphicx}
\usepackage[colorlinks=true, linkcolor=blue, citecolor=blue, urlcolor=blue]{hyperref}
\usepackage{tikz}
\usepackage{lipsum}
\usepackage{enumitem}
\usepackage{mathtools}
\usepackage{bm}
\usepackage{bbm}
\usepackage{stmaryrd}
\usepackage{dsfont}
\usepackage{float}
\usepackage{caption}
\usepackage{subcaption}
\usepackage{booktabs}
\usepackage{multirow}
\usepackage{algorithm}
\usepackage{algpseudocode}
\usepackage{color}
\usepackage{xcolor}

\theoremstyle{plain}
\newtheorem{theorem}{Theorem}[section]
\newtheorem{lemma}[theorem]{Lemma}
\newtheorem{proposition}[theorem]{Proposition}
\newtheorem{corollary}[theorem]{Corollary}
\newtheorem{remark}[theorem]{Remark}
\newtheorem{definition}[theorem]{Definition}
\newcommand{\R}{\mathbb R}
\theoremstyle{definition}

\title{Non-linear Stegall's lemma and general Hamilton-Jacobi-Bellman equations on Wasserstein spaces}
\author{Charles Bertucci$^1$, Pierre-Louis Lions$^{1,2}$}
\date{}
\begin{document}

\maketitle

\begin{abstract}
We present a comparison principle for unbounded viscosity solutions to Hamilton-Jacobi equations on Wasserstein spaces of probability measures over $\mathbb{R}^d$. In addition to the use of standard techniques of viscosity solutions, our approach requires a key extension on Wasserstein spaces of a result of perturbed optimization on Banach spaces due to Stegall.
\end{abstract}

\section{Introduction}

Motivated by the study of a problem of large deviations of random matrices in an ongoing work \cite{ref1} with P.E. Souganidis, we present in this article a comparison principle for viscosity sub-super solutions to Hamilton-Jacobi-Bellman (HJB) equations on Wasserstein spaces.

The difficulties arising from the growth of the solutions on the usual proof of the comparison principle for HJB equations (in finite dimension) has been studied since the seminal work of Ishii \cite{ref2}, which followed the development of the theory of viscosity solutions by Crandall and the second author \cite{ref3}. A key argument was presented by Alvarez in \cite{ref4} for first order equations with a convex Hamiltonian, in finite dimension. Namely, he showed that when the Hamiltonian is convex in the gradient argument, we can treat arbitrary growth at $+\infty$ for the solutions, as long as they stay bounded from below for instance. It turns out that the equation at interest in \cite{ref1} has a quadratic Hamiltonian, and thus is well suited for this type of argument. This motivated us to adapt to this infinite-dimensional setting the arguments of \cite{ref4}.

Furthermore, as shall be apparent later on, to establish comparison principles for HJB  equations on non-compact sets, a perturbed optimization result is needed. In the theory of viscosity solutions on Banach spaces \cite{crandalllions}, the so-called Stegall's Lemma \cite{stegall} has proven to be of great practical use. This result is stated on Banach spaces. We explain in Section \ref{sec:stegall} how to extend this result to Wasserstein spaces, at the price of losing the linearity of the perturbation, which we show is not needed in our problem. Prior to that, the equations and main definitions are given in Section \ref{sec:prel}. We then present the equivalent of the result of Alvarez in Section \ref{sec:min} and conclude to obtain our comparison principle in Section \ref{sec:comp}.

We also mention that since the existence issue, for the equations we are about to treat, are becoming more and more standard, we do not investigate such questions here, and refer to results of \cite{bertucci2023stochastic} for results of existence of viscosity solutions through optimal control formulation.

\subsection*{Bibliographical comments}

HJB equations on spaces of measures have received a lot of attention recently. For equations which do not involve singular first order terms, the first author proposed in \cite{bertucci2023stochastic} a notion of viscosity solution when the equation is set on the space of probability measures on the torus, see also \cite{ref10} for generalizations to manifolds. In general, the difficulty arising from the lack of local compactness has been avoided up to now by adding additional regularity on the solutions, which in fact restores some local compactness, see \cite{ref11} for instance, and thus the problem at hand is in fact locally compact. Here, we adopt a different approach, motivated by cases (once again see \cite{ref1}) in which extra-regularity assumptions are not realistic nor possible.

For equations involving singular first order terms, we refer the reader to \cite{ref11,ref12}, as we shall not investigate such cases here.

\subsection*{Notation}

In all what follows, $p \in (1, \infty)$ and $p' = \frac{p}{p-1}$. The set of Borel probability measures with moment of order $p$ is denoted by $\mathcal{P}_p(\mathbb{R}^d)$. For $i \in \mathbb{N}^*$, $\pi_i$ is the projection of a tuple onto its $i$-th coordinate, i.e. $\pi_i(x_1,\dots,x_n)= x_i$. The push-forward of a measure $\mu$ by a map $T$ is denoted by $T_\# \mu$. For $\gamma \in \mathcal{P}_p(\mathbb{R}^d \times \mathbb{R}^d)$, we note 
$$
C_p(\gamma) = \left(\int_{\mathbb{R}^d \times \mathbb{R}^d} |x-y|^p \gamma(dx, dy)\right)^\frac1p.
$$
 For $\mu \in \mathcal{P}_p(\mathbb{R}^d)$, we note $M_p(\mu) = \int_{\mathbb{R}^d} |x|^p \, \mu(dx)$. We also define $\mathcal{P}_{pp'}(\mathbb{R}^d \times \mathbb{R}^d) := \{\gamma \in \mathcal{P}(\mathbb{R}^d \times \mathbb{R}^d), (\pi_1)_\# \gamma \in \mathcal{P}_p, (\pi_2)_\# \gamma \in \mathcal{P}_{p'}(\mathbb{R}^d)\}$. The set $\mathcal P_{pp'pp}(\R^{4d})$ is defined similarly. Given two probability measures $\mu$ and $\nu$ on respectively $X$ and $Y$, the set of couplings between them, that is of probability measures $\gamma$ on $X \times Y$ such that $(\pi_1)_\# \gamma = \mu$ and $(\pi_2)_\# \gamma = \nu$, is denoted by $\Pi(\mu, \nu)$. When endowed with the metric $\mathcal{W}_p$ defined by:

\[
\mathcal{W}_p(\mu, \nu) = \inf_{\gamma \in \Pi(\mu, \nu)} \left( \int_{\mathbb{R}^d \times \mathbb{R}^d} |x-y|^p \, \gamma(dx, dy) \right)^{1/p},
\]
the space $\mathcal P_p(\R^d)$ is a complete separable metric space called the $p$-Wasserstein space.
We shall also consider the quantity $\mathcal I_p$ defined on $\mathcal P_{p'}(\mathbb R^d)\times \mathcal P_p(\mathbb R^d)$ by
$$
\mathcal I_p(\xi,\mu) = \inf_{\gamma \in \Pi(\xi,\mu)}- \int_{\mathbb R^d\times \mathbb R^d}x\cdot y\,\gamma(dx,dy).
$$
We fix an atomless Polish probability space $(\Omega, \mathcal{A}, \mathcal{P})$ and we denote by $\mathbb L^p$ the set of $\mathbb{R}^d$ valued random variables on $\Omega$ of finite $p$-th order moment. It is a Banach space equipped with the norm $\|X\|_p = \left( \mathbb{E}[|X|^p] \right)^{1/p}$ where $|x|$ denotes the Euclidean norm of $x \in \mathbb{R}^d$. The law of a random variable $X$ is denoted by $\mathcal L(X)$.
Upper semi continuous and lower semi continuous shall be abbreviated in u.s.c. and l.s.c. . A point of a maximum of a given function is strongly exposed if all maximizing sequences converge toward this point. A function defined on a metric space is locally bounded from above if it is bounded from above on every bounded set. For a locally bounded function $u$, $u^*$ and $u_*$ denote respectively its upper and lower semi continuous enveloppes.

A modulus of continuity is a concave increasing function $\omega: \R_+ \mapsto \R_+$ such that $\omega(0) = 0$.

\section{Preliminary results and definitions}\label{sec:prel}

Let $H: \mathbb{R}^d \times \mathbb{R}^d \times \mathcal{P}_p(\mathbb{R}^d) \mapsto \mathbb{R}$. We are interested in the equation

\begin{equation}\label{hjb1}
U(\mu) + \int_{\mathbb{R}^d} H(x, D_\mu U(\mu)(x), \mu) \, \mu(dx) = 0 \quad \text{in } \mathcal{P}_p(\mathbb{R}^d),
\end{equation}
of unknown $U: \mathcal{P}_p(\mathbb{R}^d) \mapsto \mathbb{R}$. The derivative $D_\mu U$ refers to the horizontal, geometric, Lions or Wasserstein derivative on $\mathcal{P}_p(\mathbb{R}^d)$, which is a notion of derivative along Lagrangian paths in $\mathcal{P}_p(\mathbb{R}^d)$. More precisely, we say that $U$ is differentiable at $\mu_0 \in \mathcal{P}_p(\mathbb{R}^d)$ if there exists $\phi \in L_{\mu_0}^p:= L^p((\mathbb{R}^d,\mu_0),\mathbb{R}^d)$, such that for any $\mu \in \mathcal{P}_p(\mathbb{R}^d)$, any $\gamma \in \Pi(\mu, \mu_0)$

\[
\lim_{C_p(\gamma)\to 0} \frac{U(\mu) - U(\mu_0) - \int_{\mathbb{R}^d \times \mathbb{R}^d} \phi(x) \cdot (y - x) \, \gamma(dx, dy)}{C_p(\gamma)} = 0.
\]
In such a case, we note $\phi = D_\mu U(\mu_0)$, see \cite{refbook} for more details.

We also recall the appropriate definition of super-differentials in this case. For an u.s.c. $U: \mathcal{P}_p(\mathbb{R}^d) \mapsto \mathbb{R}$, we say that $\gamma \in \mathcal{P}_{pp'}(\mathbb{R}^d \times \mathbb{R}^d)$ belongs to the super-differential of $U$ at $\mu_0 \in \mathcal{P}_p(\mathbb{R}^d)$ if $\pi^1_\# \gamma = \mu_0$ and there exists a modulus of continuity $\omega$ such that, for any $\mu \in \mathcal{P}_p(\mathbb{R}^d)$, any $\Gamma \in \Pi(\gamma, \mu)$,

\[
U(\mu) - U(\mu_0) - \int_{(\mathbb{R}^d)^2} z \cdot (y - x) \, \Gamma(dx, dz,dy) \leq C_p((\pi_1,\pi_3)_\#\Gamma)\omega(C_p((\pi_1,\pi_3)_\#\Gamma)).
\]
In this case, we note $\gamma \in \partial^+ U(\mu_0)$. Sub-differentials are defined analogously as usual through $\gamma \in \partial^- U(\mu_0) \Leftrightarrow (Id, -Id)_\# \gamma \in \partial^+ (-U)(\mu_0)$.

Sub and super differentials on Wasserstein spaces are linked to the usual notion at the level of the lift as the next simple result explains.

\begin{proposition}\label{prop1}
Let $U: \mathcal{P}_p(\mathbb{R}^d) \mapsto \mathbb{R}$ and define $\mathcal{U}: \mathbb L^p \mapsto \mathbb{R}$ by
\[
\mathcal{U}(X) = U(\mathcal{L}(X)).
\]
Then,
\begin{enumerate}
\item If $\gamma \in \partial^+ U(\mu_0)$, then for any $(X, Z)$ of law $\gamma$,
\[
Z \in \partial^+ \mathcal{U}(X).
\]
\item If $Z \in \partial^+ \mathcal{U}(X)$, then $\mathcal{L}(X, Z) \in \partial^+ U(\mathcal{L}(X))$.
\end{enumerate}
\end{proposition}
The proof is immediate. We can now state the notion of viscosity solution we shall work with.
\begin{definition}
An u.s.c. function $U: \mathcal{P}_p(\mathbb{R}^d) \mapsto \mathbb{R}$ is a viscosity sub-solution (resp. super-solution) to \eqref{hjb1} if, for any $\gamma \in \partial^+ U(\mu)$ (resp. $\gamma \in \partial^- U(\mu)$),
\[
U(\mu) + \int_{\mathbb{R}^d \times \mathbb{R}^d} H(x, y, \mu) \, \gamma(dx, dy) \leq 0. \text{( resp. } \geq 0 \text{)}.
\]
A locally bounded function $U$ is a viscosity solution to \eqref{hjb1} if $U^*$ is a sub-solution and $U_*$ is a super-solution.
\end{definition}

\begin{remark}
From this definition, it follows from Proposition \ref{prop1} that $U$ is a viscosity solution (resp. super-solution, resp. sub-solution) to \eqref{hjb1} if and only if the lift $ \mathcal{U}$ of  $U$is a viscosity solution (resp. super-solution, resp. subsolution) to the HJB equation
\[
\mathcal{U}(X) + \mathbb{E}[H(X, \nabla \mathcal{U}(X), \mathcal{L}(X))] = 0 \quad \text{in } \mathbb L^p.
\]
\end{remark}
Somehow the notion of super-differentials on Wasserstein spaces is inherited from the one at the level of the lift. A similar notion of convexity can be pulled down, and we shall way that a set $A \subset \mathcal P_p(\mathbb R^d)$ is coupling convex if, for any $\mu,\nu \in A, \theta \in [0,1]$ and $\gamma \in \Pi(\mu,\nu)$, then $((1-\theta)\pi_1 + \theta \pi_2)_\#\gamma \in A$. This notion has the obvious property: if $A$ is coupling convex, then $\{X \in \mathbb L^p, \mathcal L(X) \in A\}$ is convex in $\mathbb L^p$.

In this paper, we are going to stay as close as possible to the formalism on $\mathcal{P}_p(\mathbb{R}^d)$ rather than working on $\mathbb L^p$, namely because a main motivation for this work is a more involved case in which this lift is not natural \cite{ref1}.

As mentioned above, the main objective of this paper is to establish the uniqueness of bounded from below viscosity solutions to \eqref{hjb1}, by means of a comparison result argument. In particular, we would like to be able to consider points of maximum of

\[
W: (\mu, \nu) \mapsto U(\mu) - V(\nu) - \frac{1}{p \varepsilon} \mathcal{W}_p^p(\mu, \nu),
\]
when $U$ is a viscosity sub-solution and $V$ is a viscosity super-solution. Clearly $W$ is not bounded a priori, and even if it was, there is no reason why it should reach its maximum. The aim of Sections 3 and 5 is to show how we can perturb the map $W$ to obtain such properties.

\section{Non-linear variants of Stegall's lemma on spaces of probability measures}\label{sec:stegall}

The aim of this section is to prove and comment the following result.

\begin{theorem}\label{thm:stegallnon}
Let $\mathcal O$ be a closed convex subset of $\mathbb{R}^d$ and $K$ a closed, coupling convex, bounded subset of $\mathcal{P}_p(\mathcal O)$. Then, for every upper bounded u.s.c. function $f: K \mapsto \mathbb{R}$, $\varepsilon > 0$, there exists $\xi \in \mathcal{P}_{p'}(\mathbb{R}^d)$ such that $\mu \mapsto f(\mu) - \mathcal I_p(\xi, \mu)$ has a strongly exposed point of maximum on $K$, and $M_{p'}(\xi) \leq \varepsilon$.
\end{theorem}
By definition, a strongly exposed point of maximum of $g:K\to \mathbb R$ is a point $x_*$ such that all sequences $(x_n)_{n \geq 0}$ such that $g(x_n)\to \sup_K g$ are such that $\lim_{n \to \infty} x_n = x_*$.
\begin{proof}
The proof relies on the usual version of Stegall's Lemma which is recalled below. Let $ E := \{X \in \mathbb L^p, \mathcal{L}(X) \in K\}$. Clearly, $E$ is convex closed and bounded in $\mathbb L^p$. Denote $\mathcal F: E \mapsto \mathbb{R}$ the lift of $f$. By applying Stegall's Lemma to $\mathcal F$ on $E$, we deduce that there exists $Y \in \mathbb L^{p'}$, $\|Y\|_{p'} \leq \varepsilon$, such that

\[
X \mapsto \mathcal F(X) + \mathbb{E}[Y \cdot X]
\]
has an exposed point of maximum on $E$, denoted by $X^*$. Denote now $\xi = \mathcal{L}(-Y)$. First, $M_{p'}(\xi) \leq \varepsilon$ holds by construction. For any $X \in E$, such that $\mathcal{L}(X) = \mathcal{L}(X^*)$, it follows that $\mathbb{E}[Y \cdot X^*] \geq \mathbb{E}[Y \cdot X]$. In particular $\mathcal{L}(-Y,X^*)$ is optimal for $\mathcal I_p(\mathcal{L}(X^*), \xi)$. Furthermore, because $X^*$ is strongly exposed, for any $\mu' \in K$, $\mu' \neq \mu^*:= \mathcal L(X^*)$,

\[
\sup_{\mathcal{L}(X') = \mu'} \mathcal F(X') + \mathbb{E}[Y \cdot X'] = f(\mu') - \mathcal I_p(\xi,\mu') < f(\mu^*) - \mathcal I_p(\xi,\mu^*).
\]
Hence $\mu^*$ is indeed a point of maximum. To show that it is strongly exposed and conclude, it is sufficient to observe that for a maximizing sequence $(\mu_n)_{n \geq 0}$, we can consider an associated sequence of random variables $(X_n)_{n \geq 0}$ such that $\forall n \geq 0$, $\mathcal{L}(X_n) = \mu_n$. Since $X^*$ is exposed, $\|X_n - X^*\|_p \to 0$ and thus $\mathcal{W}_p(\mu_n, \mu^*) \to 0$.
\end{proof}

Several comments are now in order.

\begin{enumerate}
\item The perturbation given by the theorem is not linear nor smooth. In fact it is not even differentiable in general. However it is always superdifferentiable, and the size of its superdifferential (or of its elements) is bounded by $\varepsilon$, which is sufficient for our needs.
\item In the rest of the paper, only the case $\mathcal O = \mathbb{R}^d$ will be of interest. However, we hope that future developments will make it easier to treat equations on, say, $\mathcal{P}(\mathbb{R}^+)$, which will then be direct cases of applications of our results.
\item Any closed ball of $\mathcal P_p(\mathbb R^d)$ is coupling convex, closed and bounded.
\end{enumerate}

We recall here Stegall's Lemma.

\begin{theorem}\label{thm:stegall}
Let $K$ be a closed convex bounded subset of a Banach space $X$ with the Radon-Nikodym property and consider $f: K \mapsto \mathbb{R}$ a lower bounded l.s.c. function. Then, for any $\varepsilon > 0$, $\exists x^* \in X^*$, $\|x^*\| \leq \varepsilon$, and the map $x \mapsto f(x) + \langle x^*, x \rangle$ has an exposed point of minimum on $K$.
\end{theorem}

We recall that a Banach space $X$ has the Radon-Nikodym property if every bounded set possesses slices of arbitrary small diameter. A slice of $K$ is a set of the form $\{ x \in K, \langle y,x\rangle \leq \alpha \}$ for $\alpha \in \mathbb R$ and $y \in X'$. Reflexive Banach spaces have the Radon-Nikodym property, see \cite{diestel}. A proof of Theorem \ref{thm:stegall} can be found in \cite{stegall}.

We also provide two variants of Theorem \ref{thm:stegallnon} that we shall use later on.

\begin{corollary}\label{cor1}
Let $\mathcal O$ be a closed convex subset of $\mathbb{R}^d$ and $K$ a closed, coupling convex, bounded subset of $\mathcal{P}_p(\mathcal O)$. Then, for every $\nu \in \mathcal P(\R^k)$, upper bounded u.s.c. function $f: \cup_{\mu \in K}\Pi(\nu,\mu) \mapsto \mathbb{R}$, $\varepsilon > 0$, there exists $\xi \in \mathcal{P}_{p'}(\mathbb{R}^d)$ such that $\mu \mapsto f(\mu) - \mathcal I_p(\xi, \mu)$ has a strongly exposed point of maximum on $\cup_{\mu \in K}\Pi(\nu,\mu)$, and $M_{p'}(\xi) \leq \varepsilon$.
\end{corollary}
\begin{proof}
The proof follows exactly the one of Theorem \ref{thm:stegallnon}. In this case, we consider a random variable $Z$ of law $\nu$ and we consider the lifted function $\mathcal F(X) = f(\mathcal L(Z,X))$. The rest of the arguments then follows.
\end{proof}
The previous statement may be unusual, but in what follows, we shall consider functions on sets such as $\cup_{\mu \in K}\Pi(\nu,\mu)$, which are somehow natural when considering relations like the one appearing in the definition of super-differentials. The next extension is more usual and simply state the result for functions of several variables, which needs to be stated on its own since, for instance, $\mathcal P_p(\R^d)\times\mathcal P_p(\R^d) \ne \mathcal P_p(\R^{2d})$.
\begin{corollary}\label{cor2}
Let $\mathcal O$ be a closed convex subset of $\mathbb{R}^d$ and $K$ a closed, convex along couplings, bounded subset of $\mathcal{P}_p(\mathcal O)$. Then, for every upper bounded u.s.c. function $f: K\times K \mapsto \mathbb{R}$, $\varepsilon > 0$, there exists $\xi_1,\xi_2 \in \mathcal{P}_{p'}(\mathbb{R}^d)$ such that $(\mu,\mu') \mapsto f(\mu,\mu') - \mathcal I_p(\xi_1, \mu)- \mathcal I_p(\xi_2,\mu')$ has a strongly exposed point of maximum on $K \times K$, and $M_{p'}(\xi_1),M_{p'}(\xi_2) \leq \varepsilon$.
\end{corollary}
\begin{proof}
Once again, the proof follows exactly the one of Theorem \ref{thm:stegallnon}. In this case, we consider the lifted function $\mathcal F(X,Y) = f(\mathcal L(X),\mathcal L(Y))$. Stegall's Lemma on $\mathbb L^p\times \mathbb L^p$ then yields a pair $(Y_1,Y_2)$ for the perturbation. We let $\xi_i = \mathcal L(Y_i)$ and we conclude as previously.
\end{proof}

\section{On the minimum of two sub-solutions for convex Hamilton-Jacobi equations}\label{sec:min}

In this section, we revisit a result of Alvarez \cite{ref4} who established comparison of bounded from below viscosity solutions to HJB equations in $\mathbb R^d$. Namely, we prove that, for a convex HJB equation, the minimum of two sub-solutions is also a sub-solution. In fact we directly present our new proof at the level of the Wasserstein space and leave to the interested reader to verify that a similar argument also works in finite dimensions.

The main result of this section will not directly be used in the proof of the comparison principle of Section \ref{sec:comp}, however we believe it is one of the core ideas behind it, so we state it independently to highlight it.

Given $U$ and $V$ two viscosity sub-solutions to

\begin{equation}\label{hjb}
U + \int_{\mathbb{R}^d} H(x, D_\mu U(\mu)(x), \mu) \, \mu(dx) = 0 \quad \text{on } \mathcal{P}_p(\mathbb{R}^d)
\end{equation}
for $H: \mathbb{R}^d \times \mathbb{R}^d \times \mathcal{P}_p(\mathbb{R}^d) \mapsto \mathbb{R}$, convex and differentiable in its second argument and satisfying 
\begin{equation}\label{Hdef}
\begin{aligned}
\forall R > 0&, \exists C_R > 0, \forall \mu \in \mathcal P_p(\mathbb R^d), M_p(\mu) \leq R,  \forall x,q \in \R^d,\\
& |H(x,q,\mu)| \leq C_R(1 + |q|^{p'}),
\end{aligned}
\end{equation}
\begin{equation}\label{eq:contH}
\begin{aligned}
\forall \mu&,\nu \in \mathcal P_p(\mathbb R^d), \forall x,y,q \in \R^d, \\
& |H(x,q,\mu) - H(y,q,\nu)| \leq \omega_H((1+ |q|)(|x-y| + \mathcal W_p(\mu,\nu))),
 \end{aligned}
\end{equation}
for $\omega_H$ a modulus of continuity. Note that, using \eqref{Hdef} and the convexity of $H$, we obtain as well as
$$
\begin{aligned}
\forall R > 0&, \exists C_R > 0, \forall \mu \in \mathcal P_p(\mathbb R^d), M_p(\mu) \leq R,  \forall x,q \in \R^d,\\
& |D_q H(x,q,\mu)| \leq C_R (1 + |q|^{p'-1}).
\end{aligned}
$$
We want to establish that $\min(U, V)$ is also a subsolution. Instead of working directly on the minimum, we study first a regularization of it. For $\delta \in (0, 1)$, consider $F_\delta: \mathbb{R}^2 \mapsto \mathbb{R}$ defined by

\[
F_\delta(u, v) = u - g_\delta(u - v)
\]
where $g_\delta$ is an approximation of the positive part given by

\[
g_\delta(x) = \begin{cases}
0 & \text{if } x \leq 0, \\
\frac{1-\delta}{2\delta}x^2 & \text{if } x \in [0, \delta], \\
(1-\delta)x & \text{else.}
\end{cases}
\]
We now list some properties of $F_\delta$ that we shall use below. The sequence $(F_\delta(\cdot,\cdot))_{\delta>0}$ converges locally uniformly toward the minimum as $\delta \to 0$. For all $\delta \in (0,1)$, $F_\delta$ is a $C^1$ function and, for any $u, v \in \mathbb{R}$, $\partial_u F_\delta(u,v), \partial_v F_\delta(u,v) \in (0,1)$ and $\partial_u F_\delta(u,v) + \partial_v F_\delta(u,v) = 1$. Furthermore the map $(u,v) \mapsto \partial_u F_\delta(u,v)u + \partial_v F_\delta(u,v)v$ converges locally uniformly toward the minimum as $\delta \to 0$.

The main result of this section is the following.

\begin{theorem}\label{thm:min}
Let $U$ and $V$ be two continuous viscosity sub-solutions to \eqref{hjb} which are locally bounded from above. Then $U \wedge V$ is also a viscosity sub-solution of \eqref{hjb}.
\end{theorem}

\begin{proof}
We present the proof in the case $p=2$ to simplify notation (namely when considering the super-differential of the Wasserstein distances), but this extends immediately to the general case $p \in (1,\infty)$. Let $\gamma \in \partial^+ (U \wedge V)(\mu)$. By definition, there exists a modulus of continuity $\omega$ such that $\forall \mu' \in \mathcal{P}_2(\R^d), \forall \Gamma \in \Pi(\gamma, \mu')$

\begin{equation}\label{eq:omega}
U \wedge V(\mu') \leq U \wedge V(\mu) + \int_{\R^{3d}} z \cdot (y - x) \Gamma(dx,dz, dy) +  C_2(\tilde\gamma)\omega(C_2(\tilde\gamma))
\end{equation}
where we used the notation $\tilde{\gamma} = (\pi_1, \pi_3)_\# \Gamma$. Note that, without loss of generality we can always assume that the previous inequality is strict as soon as $\mu' \neq \mu$. Hence, taking $U\wedge V(\mu')$ on the right-hand side of the inequality, the function on the right-hand side is bounded from below by $0$ (as a function of $\Gamma$ and reaches its minimum for the unique $\Gamma$ such that $C_2(\tilde \gamma) = 0$.

Let us now consider $\varepsilon > 0$, $\delta \in (0,1)$ and consider the function $\mathcal{Z}$ defined on $\mathcal{P}_\gamma := \{ \Theta\in \mathcal{P}_{2}(\mathbb{R}^{4d}), (\pi_1, \pi_2)_\# \Theta = \gamma \}$ by

\[
\begin{aligned}
\mathcal{Z}(\Theta) = &U \wedge V(\mu) - F_\delta(U(\mu'), V(\nu)) + \frac{1}{2\varepsilon} \mathcal{W}_2^2(\mu', \nu) + C_2(\tilde\gamma)\omega(C_2(\tilde\gamma))\\
 &+ \int_{\R^{3d}} z \cdot (y - x) \Gamma(dx,dz, dy)
\end{aligned}
\]
where we used the notation $(\pi_1, \pi_2, \pi_3)_\# \Theta = \Gamma$, $(\pi_1, \pi_3)_\# \Theta= \tilde{\gamma}$, $(\pi_3)_\# \Theta = \mu'$ and $(\pi_4)_\# \Theta= \nu$, which would be in force for the rest of the proof. In other words, we doubled the variable in the minimum, penalized this doubling of variable with the $2$-Wasserstein distance, and regularized the minimum with $F_\delta$.

We would like to argue as if $\mathcal{Z}$ is coercive, and then use the variants of Stegall's Lemma introduced in Section \ref{sec:stegall} to conclude the existence of minimizers. Remark that $\mathcal Z$ is coercive if it is coercive in $\mu'$ and in $\nu$. Up to taking $\omega$ with a bigger growth, we can always assume that $\mathcal{Z}$ is coercive in $\mu'$.

Assume for the moment that $V$ growth in $\nu$ is bounded by $C(1 + M_2(\nu))$ for some constant $C$, we shall come back on the general case at the end of the proof. The coercivity in $\nu$ then follows provided that $\epsilon$ is small enough. Hence, from now on, we can argue as if we are restricted to a bounded set of $\mathcal P_\gamma$. Therefore, from our assumption, $U$ and $V$ are bounded on the set of interest.
Thanks to Corollary \ref{cor1}, we know that for every $\kappa > 0$, $\exists \eta \in \mathcal P_{2}(\R^{2d})$ such that $M_{2}(\eta) \leq \kappa$ and

\[
\Theta \mapsto \mathcal{Z}(\theta) + \mathcal I_2(\eta, (\pi_3, \pi_4)_\# \Theta)
\]
has a strongly exposed point of minimum $\Theta^*$ on $\mathcal{P}_\gamma$. Consider now a sextuple of random variables $(X, X', Y, Z, \xi_1, \xi_2)$ such that

\[
\mathcal{L}(X, Z, X', Y) = \Theta^* \quad \text{and} \quad (\mathcal L((\xi_1, \xi_2), (X', Y))) \quad \text{is optimal for } \mathcal I_p(\eta, (\pi_3, \pi_4)_\# \Theta).
\]
 We denote $\mu'_* = (\pi_3)_\#\Theta^*$ and $\nu_* = (\pi_4)_\#\Theta^*$. From the minimality of $\Theta^*$, we necessary have that $\mathcal L(X',Y)$ is optimal for $\mathcal W_2(\mu'_*,\nu_*)$, and furthermore we deduce that

\begin{equation}\label{couple1}
\mathcal{L}\left(X', \frac{1}{\partial_u F_\delta(U(\mu'_*), V(\nu_*))}\left(\frac{X' - Y}{\varepsilon} + Z + \xi_1 + \varphi(X' - X)\right) \right)\in \partial^+ U(\mu'_*),
\end{equation}

\begin{equation}\label{couple2}
\mathcal{L}\left(Y, \frac{1}{\partial_v F_\delta(U(\mu'_*), V(\nu_*))}\left(\frac{Y - X'}{\varepsilon} +\xi_2 \right) \right)\in \partial^+ V(\nu_*),
\end{equation}
where $\varphi(x) = \nabla_x(|x|\omega(|x|))$, which is continuous and such that $\phi(0) = 0$. To lighten notation, we note $(X', A)$ and $(Y, B)$ the two couples of random variables whose law is taken in respectively \eqref{couple1} and \eqref{couple2}. Using the fact that $U$ and $V$ are viscosity sub-solutions, we deduce that
\[
U(\mathcal{L}(X')) + \mathbb{E}[H(X', A, \mathcal{L}(X'))] \leq 0, \quad (i)
\]
\[
V(\mathcal{L}(Y)) + \mathbb{E}[H(Y, B, \mathcal{L}(Y))] \leq 0. \quad (ii)
\]
Using the regularity assumption on $H$, we know that
\[
|H(X', B, \mathcal{L}(X')) - H(Y, B, \mathcal{L}(Y))| \leq \omega_H((1 + |B|)(|X' - Y| + \mathcal{W}_2(\mathcal{L}(X'), \mathcal{L}(Y))))).
\]
We now write the linear combination $\partial_u F_\delta(U(\mathcal{L}(X')), V(\mathcal{L}(Y))) (i) + \partial_v F_\delta(U(\mathcal{L}(X')), V(\mathcal{L}(Y)))(ii)$, use the previous estimate on $H$, Jensen's inequality and the convexity of $H$ in its second argument to obtain

\begin{equation}\label{eq2}
\begin{aligned}
 & \mathbb{E}[H(X', Z + \xi_1 + \xi_2 + \varphi(X'-X) , \mathcal{L}(X'))]+\partial_u F_\delta \cdot U(\mathcal{L}(X')) \\
 &+ \partial_v F_\delta \cdot V(\mathcal{L}(Y)) \leq \omega_H\left((1+\kappa)\delta^{-1}\left(\|X'-Y\|_2 + \frac{\|X'-Y\|_2^2}{\epsilon}\right)\right).
 \end{aligned}
\end{equation}
It now remains to pass to the limit $\kappa, \varepsilon, \delta \to 0$ to conclude. First observe that from standard arguments (e.g. Lemma 3.1 in \cite{users}), we know that $\frac{\|X'-Y\|_2^2}{\varepsilon} \to 0$.

On the other hand, because of the way we choose $\omega$, we know that as $\kappa, \varepsilon, \delta \to 0$, $X' \to X$ in $\mathbb L^2$. Indeed, this comes from: i) the previous estimate on the distance between $X'$ and $Y$, ii) the bound on $U$ and $V$ thanks to which $|F_\delta(U(\mu'),V(\nu))- U(\mu')\wedge V(\nu)|$ goes to $0$ as $\delta \to 0$, uniformly in $\epsilon$ and $\kappa$, and iii) the fact that \eqref{eq:omega} is strict if $\mu' \ne \mu$. Hence, $X', Y \to X$, once again in $\mathbb L^2$, so that $\varphi(X'- Y) \to 0$. Using the continuity of $U$ and $V$ at $\mathcal{L}(X)$, we know that $\partial_u F_\delta \cdot U(\mathcal{L}(X')) + \partial_v F_\delta \cdot V(\mathcal{L}(Y)) \to U \wedge V(\mathcal{L}(X))$. Finally, using the regularity of $H$, we obtain

\[
\begin{aligned}
& \left| H (X, Z, \mathcal{L}(X)) - H (X', Z + \xi_1 + \xi_2 + \varphi(X'- Y), \mathcal{L}(X')) \right| \\
& \leq \left| H (X, Z, \mathcal{L}(X)) - H (X', Z, \mathcal{L}(X')) \right| \\
& + \left| H (X', Z, \mathcal{L}(X')) - H (X', Z + \xi_1 + \xi_2 + \varphi(X'-Y), \mathcal{L}(X')) \right| \\
& \leq \omega_H((1 + |Z|)(|X - X'| + \mathcal{W}_2(\mathcal{L}(X), \mathcal{L}(X'))|)) \\
& + C_{\|X'\|_2}(1 + |Z| + |\xi_1 + \xi_2 + \varphi(X'-Y)| )|\xi_1 + \xi_2 + \varphi(X'- Y)|.
\end{aligned}
\]

Hence, since the expectation of the right-hand side of the previous inequality vanishes in the limit, namely because $\|X'\|_2$ stays bounded, we deduce that we can indeed pass to the limit in \eqref{eq2} to obtain the result.\\

It now remains to treat the case in which $V$ grows faster than $\mathcal W_2^2$. In order to do so, we just need to perturb the previous argument by changing slightly the doubling of variable argument. We then use $\psi(\mathcal W_2^2(\mu',\nu))$, for some function $\psi$ which grows sufficiently fast at infinity. Then, Lemma 3.1 in \cite{users} yields that $\psi(\frac{\|X'-Y\|_2^2}{\varepsilon}) \to 0$, and it suffices to remark that this also implies the convergence toward $0$ of $\frac{\|X'-Y\|^2_2}{\varepsilon}\psi(\frac{\|X'-Y\|_2^2}{\varepsilon})$, which is the term we need to control in the estimate. Hence, the result follows in the general case as well.
\end{proof}

\section{Comparison principles for unbounded solutions to Hamilton-Jacobi-Bellman equations}\label{sec:comp}

In this section, we prove the main result of the paper: a comparison principle for unbounded viscosity sub/super-solutions to 
\begin{equation}\label{hjb2}
U(\mu) + \int_{\mathbb{R}^d} H(x, D_\mu U(\mu)(x)) \, \mu(dx) = F(\mu) \quad \text{in } \mathcal{P}_p(\mathbb{R}^d).
\end{equation}
The previous equation is slightly less general than \eqref{hjb1} which was considered in the previous section and we come back on this in the second part of this section. We assume here that the assumptions of the previous section still hold, which in this separable setting takes the following form. The Hamiltonian $H$ is convex in its second argument. The source term $F$ is continuous on $\mathcal P_p(\mathbb R^d)$
\begin{equation}\label{Hdef2}
\begin{aligned}
\exists C > 0,  \forall x,q \in \R^d, \quad|H(x,q)| \leq C(1 + |q|^{p'}),
\end{aligned}
\end{equation}
\begin{equation}\label{eq:contH2}
\begin{aligned}
\forall x,y,q \in \R^d,\quad |H(x,q) - H(y,q)| \leq \omega_H((1+ |q|)(|x-y|)),
 \end{aligned}
\end{equation}
for $\omega_H$ a modulus of continuity. Once again, using \eqref{Hdef2} and the convexity of $H$, we obtain as well that
$$
\begin{aligned}
\exists C > 0,  \forall x,q \in \R^d, \quad|D_q H(x,q,\mu)| \leq C (1 + |q|^{p'-1}).
\end{aligned}
$$

\subsection{Main result}
 We shall argue by using the variant of Stegall's Lemma introduced in Section 3 and the core idea of Section 4, which is that the truncation of a sub-solution by another sub-solution is still a sub-solution. To explain how to use those two key ideas, let us anticipate slightly on the following and start to give elements of proof. Using \cite{ref3,bertucci2023stochastic}, to prove that a subsolution $U$ to \eqref{hjb1} is below a supersolution $V$ to the same \eqref{hjb1}, we would like to use the standard argument which consists in looking at points of maxima of

\[
(\mu, \nu) \mapsto U(\mu) - V(\nu) - \frac{1}{p \varepsilon} \mathcal{W}_p^p(\mu, \nu).
\]
Because neither $U$ is bounded from above or $V$ bounded from below, there is no chance a priori that such points exist. The first step consists in truncating $U$ by replacing it with $U \wedge A$ for $A \in \mathbb{R}$. Even if constants are not exactly sub-solution to \eqref{hjb1}, the previous idea still applies under some assumption on $H$ and we will show that we can argue as if $U \wedge A$ is a sub-solution. We shall also restrict ourselves to functions $V$ which satisfies some bound from below so that, for any $\mu \in \mathcal{P}_p(\R^d)$, $\nu \mapsto -V(\nu) - \frac{1}{p \varepsilon} \mathcal{W}_p^p(\mu, \nu)$ is indeed coercive. So if we are able to argue that $U \wedge A$ were coercive as well, we could indeed localize points of maximum of interest. Arguing as in Alvarez \cite{ref4}, we shall construct a sub-solution $\tilde{U}$ to \eqref{hjb1} such that

\begin{equation}\label{coerc}
\lim_{M_p(\mu) \to +\infty} \quad \frac{\tilde U(\mu)}{1+ (M_p(\mu))^\frac1p} = -\infty.
\end{equation}
Then replacing $U \wedge A$ by $(1-\theta)(U \wedge A) + \theta \tilde U$ for $\theta \in [0,1]$, we will see that we can indeed compare $(1-\theta)(U \wedge A) + \theta \tilde U$ and $V$, and thus obtain the result by letting $A \to +\infty$ and $\theta \to 0^+$.\\

To make the previous heuristic precise, we need some bound from below on the source term $F$ as follows,
$$
\forall \mu \in \mathcal P_p(\R^d), \quad F(\mu) \geq -C_H(1 +(M_p(\mu))^\frac1p).
$$
We shall also require in a first time the following structural assumption on the Hamiltonian
\begin{equation}\label{eq:H3}
\forall x,q \in \R^d, \theta \in [0,1], \quad H(x,\theta q) \leq H(x,q).
\end{equation}
Note that \eqref{eq:H3} holds if for all $x$, $H(x,\cdot)$ is minimum at $0$ since $H(x,p)$ is convex in $p$. We come back on such an assumption in Section \ref{sec:structH}.

We start with a lemma.
\begin{lemma}\label{lemma:sub}
Provided that $\lambda > 0$ is small enough $C^*$ is large enough,
\[
\tilde U(\mu) := -C^* - \lambda M_p(\mu)
\]
is a sub-solution to \eqref{hjb2} such that \eqref{coerc} holds.
\end{lemma}

\begin{proof}
Since $\tilde U$ is smooth, it suffices to compute $\forall \mu \in \mathcal{P}_p(\R^d)$,
\[
\begin{aligned}
& U(\mu) + \int_{\mathbb{R}^d} H(x, D_\mu U(\mu)(x)) \, \mu(dx) \\
& \leq -C^* -\lambda  M_p(\mu) + \int_{\mathbb{R}^d} C_H (p\lambda)^{p'}  \ |x|^{(p-1)p'} \mu(dx) + C_H \\
& \leq C_H - C^* - \left( \lambda - C_H (p\lambda)^{p'} \right) M_p(\mu) \\
& \leq F(\mu)
\end{aligned}
\]
where the last inequality holds provided that $\lambda$ is small enough and $C^*$ is large enough.
\end{proof}

We continue with a second lemma which is the analogous of Theorem \ref{thm:min} above that we need in our comparison principle. It is in this result that we use the assumption \eqref{eq:H3}.

\begin{lemma}\label{lemma:infA}
Let $U$ be a continuous locally bounded from above viscosity sub-solution to \eqref{hjb2}. Then, for any $A \in \mathbb{R}$, $U \wedge A$ is a viscosity sub-solution to \eqref{hjb2} as well.
\end{lemma}

\begin{proof}
The argument follows the line of the one of Theorem \ref{thm:min}, and the fact that the constant $A$ is not necessary a viscosity sub-solution to \eqref{hjb2} is compensated by \eqref{eq:H3}. Indeed, consider $ \gamma \in \partial^+ (U \wedge A)(\mu)$ and, as in the proof of Theorem \ref{thm:min}, $\Gamma$ a point of minimum on $\mathcal P_\gamma:=\cup_{\mu' \in \mathcal P_p(\R^d)}\Pi(\gamma,\mu')$ of
\[
\mathcal Z(\Gamma):= U(\mu)\wedge A -F_{\delta}(U(\mu'),A) + C_2(\tilde\gamma) \omega(C_2(\tilde\gamma)) + \int_{\R^{3d}}z\cdot(y-x)\Gamma(dx,dz,dy),
\]
where we used the notation $(\pi_3)_\#\Gamma = \mu'$ and $(\pi_1,\pi_3)_\#\Gamma = \tilde \gamma$. Since $U$ is locally bounded from above, we can conclude arguing that, up to the use of Stegall's Lemma that we do not detail once again here, there exist $(X_n, Z_n)$ such that $\mathcal{L}(X_n, Z_n) \to \gamma$ and $\mathcal{L}(X_n, Z_n) \in \partial^+ F_{\delta}(U(\cdot), A)(\mathcal{L}(X_n))$. Thus

\[
U(\mathcal{L}(X_n)) + \mathbb{E}\left[H\left(X_n, \frac{Z_n}{\partial_u F_{\delta}(U(\mathcal{L}(X_n)),  A)}\right)\right] \leq F(\mathcal{L}(X_n)).
\]

Hence, using \eqref{eq:H3} and the fact that $U \wedge A \leq U$, we find 
$$U \wedge A(\mathcal{L}(X_n)) + \mathbb{E}[H(X_n, Z_n)] \leq F(\mathcal{L}(X_n)).
$$
 The result follows by passing to the limit thanks to the continuity of $U$.
\end{proof}

We now prove the main result of this section.

\begin{theorem}\label{thm:last}

Let $U$ and $V$ be respectively sub and super-solutions to \eqref{hjb2}, such that both satisfy the growth condition
\[
\exists C > 0, \quad U(\mu) \geq -C(1 + (M_p(\mu)^\frac1p).
\]
Assume in addition that $U$ is continuous and locally bounded from above. Then $U \leq V$.
\end{theorem}

\begin{proof}
Let $A \in \mathbb{R}$, $\theta \in (0, 1]$ and consider the function $\tilde{U}$ given by Lemma \ref{lemma:sub}. Let $\hat U$ be defined by $\hat U(\mu) = (1-\theta)(U \wedge A)(\mu) + \theta \tilde U(\mu)$. Observe that $\hat{U}$ is a viscosity sub-solution to \eqref{hjb2}. Indeed, for any $\gamma \in \partial^+ \hat U(\mu)$, $(\pi_1, \frac{\pi_2 - \theta D_\mu \tilde U(\mu)(\pi_1)}{1 - \theta}) \in \partial^+ U(\mu)$. Hence, using the convexity of $H$ in $p$, we deduce as in the proof of Theorem \ref{thm:min} that $\hat{U}$ is indeed a viscosity sub-solution to \eqref{hjb2}. The function $W$ defined by

\[
W(\mu, \nu) = \hat{U}(\mu) - V(\nu) - \frac{1}{p \varepsilon} \mathcal{W}_p^p(\mu, \nu)
\]
is such that $W(\mu,\nu) \to -\infty$ as $M_p(\mu) + M_p(\nu) \to + \infty$. Furthermore, it is bounded from above. Hence, thanks to Corollary \ref{cor2}, for any $\delta > 0$, there exist $\xi_1,\xi_2 \in \mathcal{P}_{p'}(\mathbb{R}^d)$, such that $\mathcal{M}_p'(\xi_i) \leq \delta$ for $i = 1,2$ and the map
\[
(\mu, \nu) \mapsto W(\mu, \nu) - \mathcal I_p(\xi_1, \mu) - \mathcal I_p(\xi_2, \nu)
\]
has a strongly exposed point of maximum on $\mathcal{P}_p(\mathbb{R}^d) \times \mathcal{P}_p(\mathbb{R}^d)$. Let $(\mu,\nu)$ be this point of maximum, $\gamma$ an associated optimal coupling for $\mathcal{W}_p(\mu, \nu)$, and $(X, Y, Z_1, Z_2)$ a set of random variables whose law is such that $\mathcal{L}(X, Y) = \gamma$, $\mathcal{L}(X, Z_1)$ is optimal for $\mathcal I_p(\xi, \mu)$ and $\mathcal{L}(Z_2, Y)$ is optimal for $\mathcal I_p(\xi_2, \nu)$. We then deduce, using the viscosity solution properties that
\[
\hat{U}(\mu) + \mathbb{E}\left[H\left[X, \frac{X-Y}{\varepsilon} |X - Y|^{p-2}  + Z_1\right]\right] \leq F(\mu)
\]
\[
V(\nu) + \mathbb{E}\left[H\left(Y, \frac{X - Y}{\varepsilon} |X - Y|^{p-2} - Z_2\right)\right] \geq F(\nu).
\]
Furthermore, by construction of $(\mu,\nu)$, there exists $C > 0$ which depends on $\hat{U}$ and $V$, but not on $\varepsilon$ or $\delta$ such that $\sup (\hat U - V) \leq \hat U(\mu) - V(\nu) + C (\delta)^\frac{1}{p'}$. This yields

\[
\begin{aligned}
\sup(\tilde{U} - V) \leq &\mathbb{E}\left[H\left(Y, \frac{X - Y}{\varepsilon} |X - Y|^{p-2} - Z_2\right) - H\left(X, \frac{X - Y}{\varepsilon} |X - Y|^{p-2} + Z_1\right)\right]\\
&+ F(\mu) - F(\nu)  + C (\delta)^\frac{1}{p'}
\end{aligned}
\]
We now use the regularity of the Hamiltonian to estimate, for $x, y, q, z_1, z_2 \in \mathbb{R}^d$
\[
\begin{aligned}
\left| H(y, q - z_2) - H(x, q + z_1) \right| \leq \,&\omega\left((1 + |q| + |z_1|) |x - y|\right) \\
&+ C\left(1 + |q| + |z_1| + |z_2|\right)^{p'-1}\left(|z_1| + |z_2|\right).
\end{aligned}
\]
Integrating yields

\[
\begin{aligned}
\sup (\hat{U} - V) \leq& C (\delta)^\frac{1}{p'} + F(\mu) - F(\nu) + \omega\left( \mathbb{E}\left[ \left(1 + \frac{|X - Y|^p}{\varepsilon}\right) \right] + \mathbb{E}\left[ |Z|^{p'} \right]^{\frac{1}{p'}} \mathbb{E}\left[ |X - Y|^p \right] \right)\\
&+ C \mathbb{E}\left[ |Z_1 + Z_2|^{p'} \right]^{\frac{1}{p'}} \mathbb{E}\left[ \left( 1 + \frac{|X - Y|^{p - 1}}{\varepsilon} + |Z_1| + |Z_2|^p \right) \right]^{1/p}
\end{aligned}
\]
Using that uniformly in $\delta$, $\mathbb{E}\left[ |X - Y|^p \right] = o(\varepsilon)$, we obtain passing first to the limit $\delta \to 0$, then $\varepsilon \to 0$, and using the continuity of $F$, that $\sup (\hat{U} - V) \leq 0$. Since $\theta \in (0, 1)$ and $A \in \mathbb{R}$ are arbitrary, we recover that $U \leq V$.
\end{proof}

\subsection{Extensions}

We briefly describe how to adapt the previous result to more general cases.

\subsubsection{Structure of the Hamiltonian}\label{sec:structH}
The convexity of the Hamiltonian is of course a crucial requirement for what we just presented, just as the growth up to the power $p'$, because of the mean field setting. Continuity requirements \eqref{eq:contH} are also fundamental to deal with various perturbations that we introduce to create point of maximum. Here, we comment the assumption \eqref{eq:H3} as well as the separated structure of \eqref{hjb2} compared to \eqref{hjb1}.\\

Condition \eqref{eq:H3} helps to conclude the proof of Lemma \ref{lemma:infA} despite the fact that constants are not sub-solutions to \eqref{hjb2} in general. This argument is crucial in our proof. Following \cite{ref4}, we can reduce more general cases to this one. 

For instance, if $H$ depends only on $p$, we can assume that the condition holds without loss of generality, by replacing $H(p)$ with $\tilde H(p) := H(p+p_0)$ where $p_0$ is a point of minimum of $H$. We then need the changes $\tilde U(\mu) := U(\mu) - \int_{\mathbb R^d}p_0\cdot x \mu(dx)$, $\tilde V(\mu) := V(\mu) - \int_{\mathbb R^d}p_0\cdot x \mu(dx)$ and $\tilde F(\mu) := F(\mu) - \int_{\mathbb R^d}p_0\cdot x \mu(dx)$.

 When $H$ depends on $x$ or on $\mu$, then assumptions, under which such generalizations are possible, exist but do not seem natural, so we do not comment on them further.\\

Under condition \eqref{eq:H3}, the fact that \eqref{hjb1} has been replaced by \eqref{hjb2} does not play a particular role, and we leave to the interested reader to verify that Theorem \ref{thm:last} holds under usual conditions for \eqref{hjb1}.

\subsubsection{Time dependent equations}

The study of HJB equations of the form

\begin{equation}\label{hjbt}
\partial_t U + \int_{\mathbb{R}^d} H(x, D_\mu U) \, \mu(dx) = F(\mu) \quad \text{in } \mathcal{P}_p(\mathbb{R}^d).
\end{equation}
 is in fact easier than the one we just presented. In this case, modulo the assumption $F \geq 0$, constants are sub-solutions because of the absence of zero order terms. We thus state without proof the following.

\begin{theorem}
Under the assumption of Theorem \ref{thm:last} on the Hamiltonian and $F \geq 0$, if $U$ and $V$ are continuous and respectively viscosity sub and super-solutions to \eqref{hjbt} such that $U$ is locally bounded from above and $\exists C > 0$, such that for all $t \geq 0, \mu \in \mathcal P_p(\R^d)$,
\[
U(t, \mu) \geq -C(1 + (M_p(\mu))^\frac1p), \quad V(t, \mu) \geq -C(1 + (M_p(\mu))^\frac1p),
\]
then $U- V \leq \|(U(0,\cdot) - V(0,\cdot))_+\|_\infty$.
\end{theorem}

\section*{Acknowledgements}
The authors are supported by the Chair FDD (Institut Louis Bachelier). The first author is supported by the ERC through the project PaDiESeM and by the Lagrange Mathematics and Computing Research Center.

\bigskip
\noindent$^1$: CEREMADE, CNRS UMR7534, Universit\'e Paris Dauphine-PSL.\\
$^2$: Coll\`ege de France.

\end{document}